\documentclass[11pt,letterpaper]{article}
\usepackage[T1]{fontenc}
\usepackage{lmodern}
\usepackage[margin=1in]{geometry}
\usepackage{microtype}
\usepackage{amsmath,amssymb,amsthm,mathtools}
\usepackage[hidelinks]{hyperref}
\usepackage{enumitem}
\newtheorem{theorem}{Theorem}[section]
\newtheorem{lemma}[theorem]{Lemma}
\newtheorem{proposition}[theorem]{Proposition}
\newtheorem{corollary}[theorem]{Corollary}
\theoremstyle{definition}
\newtheorem{definition}[theorem]{Definition}
\newtheorem{remark}[theorem]{Remark}
\DeclareMathOperator{\asdim}{asdim}
\DeclareMathOperator{\rk}{rk}

\DeclareMathOperator{\im}{im}
\DeclareMathOperator{\Tor}{Tor}
\newcommand{\NN}{\mathbb N}
\newcommand{\ZZ}{\mathbb Z}
\newcommand{\QQ}{\mathbb Q}
\newcommand{\gp}{\mathrm{gp}}
\newcommand{\cP}{\mathcal P}
\newcommand{\cU}{\mathcal U}
\newcommand{\cV}{\mathcal V}
\newcommand{\act}{\curvearrowright}
\newcommand{\norm}[1]{\lVert #1\rVert_\infty}

\title{Finite Borel asymptotic dimension of bounded-to-one\texorpdfstring{\\}{ }commutative monoid actions}
\author{Ruijun Wang}
\date{}

\begin{document}
\maketitle
\vspace{-1.5em}
\begin{abstract}
We prove that every bounded-to-one action of a finitely generated
commutative monoid has finite asymptotic dimension, without a freeness
assumption. If the group completion has torsion-free rank $r$, we
give an explicit upper bound $(3^{r+1}-3)/2$. This answers a question of
Shinko, Weilacher and Yu, and extends a theorem of theirs.
\end{abstract}

\section{Introduction}

A \emph{countable Borel equivalence relation} is a Borel equivalence
relation on a standard Borel space whose classes are countable. Such a
relation $E$ is \emph{hyperfinite} if
\[
 E=\bigcup_{n\in\NN}E_n,
 \qquad E_n\subseteq E_{n+1},
\]
where each $E_n$ is a Borel equivalence relation with finite classes.
In \cite{weiss1984}, Weiss asked whether every Borel action of a countable
amenable group has hyperfinite orbit equivalence relation. This has been
proved for several classes of amenable groups.

A Borel graph is hyperfinite if its connectedness relation is hyperfinite, equivalently, it is an increasing union of countably many component-finite Borel graphs. In the seminal paper \cite{KST1999}, Kechris, Solecki and Todorcevic initiated the study of descriptive combinatorics. In \cite{GJ2015}, Gao and Jackson developed the rectangular partition method for Schreier graphs of countable abelian groups actions, and they proved that the graph is hyperfinite. In \cite{SchneiderSeward2024}, Schneider and Seward extended this method and proved the countable Borel equivalence relation by action of locally nilpotent group is hyperfinite. In \cite{CJMST}, Conley, Jackson, Marks, Seward and Tucker-Drob also followed this idea and developed Borel asymptotic dimension of locally finite graphs, and they proved hyperfiniteness for polycyclic group actions. In \cite{BY}, Bernshteyn and Yu proved the hyperfiniteness of Borel graphs of polynomial growth.

Borel asymptotic dimension is the Borel version of Gromov's asymptotic dimension, see \cite{Gro93}. A locally finite Borel graph of finite Borel asymptotic dimension is hyperfinite, see \cite{CJMST}. In \cite{SWY}, Shinko, Weilacher and Yu proved that a bounded-to-one free action of a finitely generated commutative monoid has finite Borel asymptotic dimension, see \cite[Theorem 1.6]{SWY}, and they also proved hyperfiniteness without a freeness assumption, see \cite[Theorem 1.5]{SWY}. Then they asked if it is true if we remove freeness assumption. In \cite{GHZ}, Gao, Hu and Zou, show an independent proof that the graph of bounded-to-one commuting Borel functions on the free part has finite Borel asymptotic dimension, see \cite[Theorem 1.2]{GHZ}. And independently in \cite{Wang}, the author of current paper showed an alternative proof that the non-free part is hyperfinite.

In this paper, we show that we can remove the freeness assumption.

\begin{theorem}\label{thm:main}
Let $M$ be a finitely generated commutative monoid of rank $r$. Every
bounded-to-one action $M\act X$ satisfies
\begin{equation}\label{eq:main-bound}
  \asdim(M\act X)\le N_r,
  \qquad
  N_r:=\sum_{j=1}^{r}3^j=\frac{3^{r+1}-3}{2}.
\end{equation}
\end{theorem}

By \cite[Corollary~5.16]{SWY} of Shinko, Weilacher and Yu, for a finitely generated commutative monoid and a bounded-to-one Borel
action,
\[
 \asdim_B(M\act X)=\asdim(M\act X).
\]

\begin{corollary}\label{cor:borel}
Every bounded-to-one Borel action of a finitely generated commutative
monoid of rank $r$ has Borel asymptotic dimension at most
$(3^{r+1}-3)/2$.
\end{corollary}

In particular, a graph generated by $d$ commuting bounded-to-one Borel
functions has finite Borel asymptotic dimension, by regarding the functions
as an $\NN^d$-action. The same applies to its nonfree part, which is Borel
and forward invariant. This answers Question~1.4 in \cite{Wang} by the author of current paper, as
well as the bounded-to-one finiteness portions of
\cite[Questions~1.8 and~5.17]{SWY}.

The exponential bound is only an estimate, the argument does not prove the proposed
sharp bound $\asdim(M\act X)\le\rk(M)$.
\begin{corollary}\label{cor:countable-hyperfinite}
Let $M$ be a countable commutative monoid and let $M\act X$ be a
bounded-to-one Borel action on a standard Borel space. Then
$E_M^X$ is hyperfinite.
\end{corollary}


This paper is a collaboration with artificial intelligence. The AI model is gpt 6. The AI proves the main theorem under human suggestions and writes the main body of early drafts. The author writes the abstract and induction and improves readability. The author has reviewed all AI-generated content.


\section{Preliminaries}

All monoids in this paper are commutative and have an identity. We write
the monoid operation additively, but write $mx$ for the action of $m$ on
$x$; thus $(m+n)x=m(nx)$. An action $M\act X$ is
\emph{bounded-to-one} if, for each $m\in M$, there is a finite $k_m$ such
that every fiber of $x\mapsto mx$ has size at most $k_m$. The bound is
allowed to depend on $m$.

Let $M^{\gp}$ be the group completion and put
\[
  \rk(M)=\dim_{\QQ}(M^{\gp}\otimes_{\ZZ}\QQ).
\]
For a finite generating set of $M$, the undirected Schreier graph has an
edge between $x$ and $sx$ for each generator $s$, with loops discarded.
Its path metric is an extended metric: distinct components are at distance
$\infty$. All diameter bounds below are uniform over all components.
Changing the finite generating set does not change asymptotic dimension.


Fix a finite generating set $S$ for a monoid $P$. Let $\ell_S(p)$ be its
word length, with $\ell_S(0)=0$, and let $d_X$ be the Schreier metric of a
$P$-set $X$. We use the following partition definition.

\begin{definition}[\cite{CJMST}]
An extended metric space $X$ has $\asdim X\le n$ if, for every integer
$s\ge1$, it has a partition into uniformly bounded sets such that each
closed ball $B_X(x,s)$ meets at most $n+1$ parts.
\end{definition}

All extended metrics in this paper take values in $\NN\cup\{\infty\}$. An $s$-chain is a finite sequence with successive
distances at most $s$. An $s$-component of a set is a maximal subset in
which any two points can be joined by an $s$-chain in that set.

\begin{lemma}[Partitions and colored covers]\label{lem:colored}
For an integer $n\ge0$, the following are equivalent:
\begin{enumerate}[label=(\roman*),itemsep=0.15em,topsep=0.3em]
\item $\asdim X\le n$ in the partition definition.
\item For every integer $s\ge1$, $X$ is a union of $n+1$ sets whose
$s$-components have uniformly bounded diameter.
\item For every integer $s\ge1$, $X$ has a uniformly bounded cover by
$n+1$ families, each family consisting of sets at pairwise distance
strictly greater than $s$.
\end{enumerate}
\end{lemma}

\begin{proof}
Conditions (ii) and (iii) are equivalent by taking components or unions of
families. Given (ii) at scale $2s$, assign overlapping color sets to their
first color and partition each color into $2s$-components. A ball of
radius $s$ meets at most one component of each color, proving (i).

Conversely, take a partition as in (i) at scale $(n+1)s$, and choose one
representative in each part. Let $f$ map a point to its representative.
There is $L<\infty$ with $d(x,f(x))\le L$, and every ball of radius
$(n+1)s$ has at most $n+1$ images under $f$. Set
\[
 U_i=\{x:f[B(x,is)]=f[B(x,(i+1)s)]\},\qquad 0\le i\le n.
\]
These sets cover $X$: among the $n+2$ nested, nonempty image sets from
radii $0,s,\ldots,(n+1)s$, two consecutive sets must agree.

If $x,y\in U_i$ and $d(x,y)\le s$, then
$B(y,is)\subseteq B(x,(i+1)s)$ and the reverse inclusion with $x,y$
interchanged gives
$f[B(x,is)]=f[B(y,is)]$. This equality propagates along an $s$-chain in
$U_i$. For any $x,y$ in the same component, there is
$w\in B(y,is)$ with $f(w)=f(x)$, whence
\[
 d(x,y)\le d(x,f(x))+d(f(w),w)+d(w,y)\le2L+is.
\]
Thus (ii) holds.
\end{proof}

\begin{lemma}[Finite unions]\label{lem:finite-union}
If $X=A\cup B$ and both subspaces have asymptotic dimension at most $n$,
then $\asdim X\le n$.
\end{lemma}

\begin{proof}
Fix $s\ge1$. By Lemma~\ref{lem:colored}, choose a cover of $A$ by families
$\cU_0,\ldots,\cU_n$, each $s$-disjoint, with all sets of diameter at most
$D$. Next cover $B$ by families $\cV_0,\ldots,\cV_n$, each
$(D+3s)$-disjoint, with a common diameter bound $E$.

For each color $i$ and $V\in\cV_i$, enlarge $V$ by adjoining every
$U\in\cU_i$ at distance at most $s$ from $V$. A set $U$ cannot be adjoined
to two different $V,V'$: that would imply
$d(V,V')\le D+2s$, a contradiction. Retain, as separate sets, all members
of $\cU_i$ that were not adjoined. The enlarged sets have diameter at most
$E+2D+2s$.

The resulting family is still $s$-disjoint. Different original $U$-sets
are $s$-disjoint. An unattached $U$ is more than $s$ from every original
$V$. Finally, if $U$ is attached to $V'$ and $V\ne V'$, then
\[
 d(V,U)\ge d(V,V')-D-s>s.
\]
These observations check all pairs of pieces in distinct enlarged sets.
Integer-valued distances ensure that the same strict separation holds
for their unions. The new $n+1$ families cover $A\cup B$, so
Lemma~\ref{lem:colored} applies. Induction gives the finite-union result.
\end{proof}

A map $F:X\to Y$ between extended metric spaces is a
\emph{quasi-isometry} if there are finite constants $a\ge1$, $b\ge0$
and $R\ge0$ such that, for all $x,x'\in X$,
\[
 \begin{aligned}
 d_Y(Fx,Fx')\le a\,d_X(x,x')+b,\\
 d_X(x,x')\le a\,d_Y(Fx,Fx')+b,
 \end{aligned}
\]
and, for every $y\in Y$, there is $x\in X$ with $d_Y(Fx,y)\le R$.
The last condition is \emph{coarse surjectivity}. Two spaces are
\emph{quasi-isometric} if such a map exists. The metric inequalities
alone define a quasi-isometric embedding and do not imply equality of
asymptotic dimensions.

\begin{lemma}[Quasi-isometry invariance]\label{lem:qi}
Quasi-isometric extended metric spaces have the same asymptotic dimension.
\end{lemma}

\begin{proof}
Suppose $a\ge1$ and $b\ge0$ and $F:X\to Y$ satisfies the above condition.

Pull back a uniformly bounded partition of $Y$ at scale
$\lceil as+b\rceil$. Its nonempty inverse images are uniformly bounded
by the lower inequality, and an $s$-ball in $X$ meets no more parts than
the corresponding ball in $Y$. Thus $\asdim X\le\asdim Y$. If $F$ is a
quasi-isometry, its image is uniformly dense; choosing a uniformly close
preimage for each point of $Y$ produces a coarse inverse satisfying
analogous inequalities. The reverse bound follows in the same way.
\end{proof}

\section{Metric conventions and coarse reductions}\label{sec:metric}

\begin{lemma}[Common futures and forward images]\label{lem:future}
For every action of a finitely generated commutative monoid,
\begin{equation}\label{eq:future}
 d_X(x,y)=\min\{\ell_S(a)+\ell_S(b):a,b\in P,\ ax=by\},
\end{equation}
where the minimum of the empty set is $\infty$. Consequently, a
forward-invariant subset $A\subseteq X$ is isometrically embedded when
equipped with its own Schreier metric. For every $c\in P$, the inclusion
$cA\hookrightarrow A$ is a quasi-isometry.
\end{lemma}

\begin{proof}
An equality $ax=by$ gives a path through the common future, proving one
inequality. Conversely, write each edge of a path from $x$ to $y$ as
$a_i x_i=b_i x_{i+1}$, where one of $a_i,b_i$ is a generator and the other
is $0$. Commutativity gives
$(\sum_i a_i)x=(\sum_i b_i)y$, with total word length at most the path
length. This proves~\eqref{eq:future}.

If $x,y\in A$, both forward paths to a common future remain in $A$, so the
intrinsic and restricted metrics agree. Finally, $cA$ is forward invariant
and $d_X(x,cx)\le\ell_S(c)$ for every $x\in A$. Its isometric inclusion
therefore has uniformly dense image.
\end{proof}

The common-future lemma is also proved for $\NN^d$-actions, in \cite[Lemmas~4.1--4.2]{Wang}. Note that none of these metric facts
requires injectivity of the action maps.

\begin{lemma}[A bounded equivalence quotient]\label{lem:quotient}
Let $P\act X$ be bounded-to-one. Suppose that $E$ is a $P$-invariant
equivalence relation, every $E$-class has $d_X$-diameter at most $D$, and
every class has at most $C$ points. Give $Z=X/E$ its induced action and
Schreier metric. Then the quotient map $\pi:X\to Z$ is a quasi-isometry,
and the action on $Z$ is bounded-to-one.
\end{lemma}

\begin{proof}
Edges descend, so $d_Z(\pi x,\pi y)\le d_X(x,y)$. A path of length $n$
in $Z$ can be lifted edge by edge, moving at most $D$ within a class
between successive lifted edges and at the endpoints. Hence
\begin{equation}\label{eq:quotient-metric}
 d_X(x,y)\le (D+1)d_Z(\pi x,\pi y)+D.
\end{equation}
This also shows that distinct components are not identified.

If $p:X\to X$ is at most $k_p$-to-one, the inverse image of a class of
size at most $C$ contains at most $k_pC$ points. Thus at most $k_pC$
classes can map to any prescribed class under the induced map on $Z$.
\end{proof}

\section{Uniform cancellation and rank-lowering collisions}\label{sec:cancellation}

\begin{lemma}[Hilbert basis theorem, \cite{Stacks}]\label{Hilbert basis theorem}
   Let $M$ be a finitely generated commutative monoid and $k$ a commutative Noetherian ring. Then the monoid algebra $k[M]$ is Noetherian. Consequently, for every $z\in k[M]$, there exists $N\geq 0$ such that
$$\operatorname{Ann}(z^n)=\operatorname{Ann}(z^N)\text{ for all }n\geq N$$
where $\operatorname{Ann}(z^n)=\{r\in k[M]:rz^n=0\}$.
\end{lemma}

A \emph{congruence} $\theta$ on $P$ is an equivalence relation preserved by addition \begin{center}$a\ \theta\ b\ \Longrightarrow\ (a+c)\ \theta\ (b+c)$ for all $c$.\end{center} If
$\theta$ is a congruence, its \emph{cancellative closure} is
\begin{equation}\label{eq:can}
 a\mathrel{\theta^{\mathrm{can}}}b
 \quad\Longleftrightarrow\quad
 \text{there is }u\in P\text{ with }(a+u)\mathrel\theta (b+u).
\end{equation}
It is again a congruence. The quotient $P/\theta^{\mathrm{can}}$ is the
image of $P/\theta$ in its group completion: the usual construction of the
group completion identifies two monoid elements exactly when they become
equal after adding the same element.

The next lemma upgrades the point-dependent witness $u$ in
\eqref{eq:can} to one translation that works for every pair.

\begin{lemma}[Uniform cancellation]\label{lem:uniform}
If $P$ is finitely generated and $\theta$ is a congruence, there is $c\in P$
such that
\begin{equation}\label{eq:uniform}
 a\mathrel{\theta^{\mathrm{can}}}b
 \quad\Longrightarrow\quad
 (a+c)\mathrel\theta (b+c)
 \qquad(a,b\in P).
\end{equation}
Consequently, if $P\act A$ satisfies all relations in $\theta$, the action
on $cA$ factors through $P/\theta^{\mathrm{can}}$.
\end{lemma}

\begin{proof}
Choose generators $g_1,\ldots,g_d$ and put $h=g_1+\cdots+g_d$. Consider
the monoid algebra
\[
 R=\QQ[P/\theta].
\]
It has basis $e_{[p]}$, indexed by the $\theta$-classes, and multiplication
$e_{[p]}e_{[q]}=e_{[p+q]}$. It is a quotient of a polynomial algebra in $d$
variables, hence is Noetherian.

Put $z=e_{[h]}$. By Lemma \ref{Hilbert basis theorem}, the ascending chain of ideals
\[
 \operatorname{Ann}(1)\subseteq\operatorname{Ann}(z)
 \subseteq\operatorname{Ann}(z^2)\subseteq\cdots
\]
stabilizes, say at $\operatorname{Ann}(z^N)$. Suppose
$(a+u)\mathrel\theta(b+u)$. Write $u=\sum_i n_i g_i$ and choose
$n\ge N$ with $n\ge n_i$ for every $i$. Define
\[
 v=\sum_i(n-n_i)g_i\in P.
\]
Then $u+v=nh$ as an identity in $P$. This uses no subtraction or
cancellation in the monoid. Translating the original relation by $v$
gives $(a+nh)\mathrel\theta(b+nh)$, hence
\[
 z^n(e_{[a]}-e_{[b]})=0.
\]
Since $n\ge N$, stabilization gives $z^N(e_{[a]}-e_{[b]})=0$.
The basis vectors $e_{[p]}$ are distinct, so
$(a+Nh)\mathrel\theta(b+Nh)$. Take $c=Nh$.
For $y=cx$, equation~\eqref{eq:uniform} gives $ay=by$ whenever
$a\mathrel{\theta^{\mathrm{can}}}b$, proving the last assertion.
\end{proof}

\begin{corollary}[Cancellation reduction]\label{cor:cancellation}
Every bounded-to-one action of a finitely generated commutative monoid $M$
is quasi-isometric to a bounded-to-one action of
$\im(M\to M^{\gp})$. This image is cancellative and has rank $\rk(M)$.
\end{corollary}

\begin{proof}
Apply Lemma~\ref{lem:uniform} with $\theta$ equal to the equality relation. Pass from
$X$ to the forward image $cX$ and use Lemma~\ref{lem:future}. Quotient
generators act by restrictions of the original maps, so their fibers
remain bounded. The two actions on $cX$ have exactly the same Schreier
graph when the generating set is passed to the quotient.
\end{proof}

\begin{lemma}[One collision lowers rank]\label{lem:collision}
Let $P$ be finitely generated and cancellative, let $\Gamma=P^{\gp}$,
and let $P\act X$ be bounded-to-one. For $p,q\in P$, set
\[
 A_{p,q}=\{x\in X:px=qx\}.
\]
With its restricted Schreier metric, $A_{p,q}$ is quasi-isometric to a
bounded-to-one action of
\begin{equation}\label{eq:collision-monoid}
 Q_{p,q}=\im\bigl(P\longrightarrow\Gamma/\ZZ(p-q)\bigr).
\end{equation}
If $\Gamma\cong\ZZ^r$ and $p\ne q$, then $\rk(Q_{p,q})=r-1$.
\end{lemma}

\begin{proof}
The set $A_{p,q}$ is forward invariant. Its action satisfies the congruence
$\theta$ generated by $p=q$. The group completion of $P/\theta$ is
$\Gamma/\ZZ(p-q)$, by its universal property. Thus
$P/\theta^{\mathrm{can}}$ is the monoid in
\eqref{eq:collision-monoid}. Lemmas~\ref{lem:uniform} and
\ref{lem:future} give the quasi-isometry and preserve bounded fibers.
If $\Gamma$ is torsion-free and $p\ne q$, the nonzero element $p-q$
generates a rank-one subgroup, so the quotient has rank $r-1$.
\end{proof}

The torsion-free hypothesis in the last sentence matters. A relation
whose difference is torsion need not lower rank. We remove that issue
before beginning the induction.

\section{Removing torsion}\label{sec:torsion}

We first record a useful fact about positive translates of finite sets.

\begin{lemma}[Finite translates]\label{lem:translate}
Let $P$ be a submonoid of an abelian group $\Gamma$ and suppose that $P$
generates $\Gamma$ as a group. For every finite $F\subseteq\Gamma$, there
is $b\in P$ such that $b+F\subseteq P$.
\end{lemma}

\begin{proof}
Write each $f\in F$ as $a_f-c_f$ with $a_f,c_f\in P$, and take
$b=\sum_{f\in F}c_f$. Then
$b+f=a_f+\sum_{g\in F\setminus\{f\}}c_g\in P$.
\end{proof}

\begin{proposition}[Torsion reduction]\label{prop:torsion}
Let $P$ be a finitely generated cancellative commutative monoid. Every
bounded-to-one action $P\act X$ is quasi-isometric to a bounded-to-one
action of
\[
 \overline P=\im\bigl(P\longrightarrow P^{\gp}/\Tor(P^{\gp})\bigr).
\]
In particular, this reduction preserves rank and makes the group
completion torsion-free.
\end{proposition}

\begin{proof}
Write $\Gamma=P^{\gp}$ and $T=\Tor(\Gamma)$. The group $T$ is finite.
Choose $b\in P$ with $b+T\subseteq P$, and define
\[
 F_t(x)=(b+t)x\qquad(t\in T).
\]
Let $E$ be the equivalence relation generated by
\begin{equation}\label{eq:torsion-edges}
 F_0(z)\mathrel E F_t(z)\qquad(z\in X,\ t\in T).
\end{equation}
It is $P$-invariant by commutativity.

We verify uniform diameter and cardinality bounds. For an elementary pair
$x=F_0(z)$ and $y=F_t(z)$, the identities
\begin{equation}\label{eq:torsion-identities}
 F_u(x)=F_{u-t}(y)\qquad(u\in T)
\end{equation}
hold simultaneously. Reversing an edge reverses the index shift, and
composing edges adds the shifts. Along an arbitrary finite $E$-chain we
therefore obtain
\begin{equation}\label{eq:torsion-chain}
 xEy\quad\Longrightarrow\quad F_0(x)=F_t(y)
 \text{ for some }t\in T.
\end{equation}
The fact that~\eqref{eq:torsion-identities} holds for \emph{every} $u$ is
what permits this composition.

Fix a finite generating set $S$ for $P$. Equation~\eqref{eq:torsion-chain}
and the common-future formula give
\[
 d_X(x,y)\le \ell_S(b)+\max_{t\in T}\ell_S(b+t)=:D
 \quad\text{whenever }xEy.
\]
If $F_t$ is at most $k_t$-to-one, the same equation gives
\[
 [x]_E\subseteq\bigcup_{t\in T}F_t^{-1}(\{F_0(x)\}),
 \qquad |[x]_E|\le\sum_{t\in T}k_t=:C.
\]
Lemma~\ref{lem:quotient} now applies to $Z=X/E$. Thus $X$ and $Z$ are
quasi-isometric and the induced $P$-action is bounded-to-one.

Pass next to the forward image $Y=bZ$, which is quasi-isometric to $Z$ by
Lemma~\ref{lem:future}. The action on $Y$ factors through $\overline P$.
Indeed, if $p-q=t\in T$ and $y=\pi(bx)$, then
\[
 \begin{aligned}
 py&=\pi((b+p)x)
     =\pi((b+t)(qx))\\
   &=\pi(b(qx))=qy,
 \end{aligned}
\]
where the third equality is one of the generating relations
\eqref{eq:torsion-edges}. The quotient generators again act by restrictions
of the original induced maps, so bounded-to-one is preserved. Finally,
$\overline P$ generates $\Gamma/T$, proving the rank assertion.
\end{proof}

\begin{corollary}\label{cor:lattice-reduction}
A bounded-to-one action of a finitely generated commutative monoid of
rank $r$ is quasi-isometric to a bounded-to-one action of a finitely
generated submonoid $P\subseteq\ZZ^r$ that generates $\ZZ^r$ as a group.
\end{corollary}

\begin{proof}
Combine Corollary~\ref{cor:cancellation} with
Proposition~\ref{prop:torsion} and identify the resulting torsion-free
finitely generated abelian group with $\ZZ^r$.
\end{proof}

\section{Finite-scale control on locally free points}\label{sec:local}

This section proves the local statement needed for induction. The
construction uses only a finite coloring and a greedy selection of lattice
markers. No bound on the number of colors will enter the dimension bound.

For an integer $R\ge0$, put
\[
 C_R=\{v\in\ZZ^r:\norm v\le R\}.
\]
For $R\ge1$, let $H_R$ be the graph on $\ZZ^r$ in which distinct vertices
are adjacent when their $\ell^\infty$-distance is at most $2R$. Its degree
is $(4R+1)^r-1$.

\begin{lemma}[A finite-window lattice rule]\label{lem:lattice-rule}
Let $r,R\ge1$ and $q\ge(4R+1)^r$. There is a function
\[
 \mathcal A:\{1,\ldots,q\}^{C_W}\longrightarrow C_{2R},
 \qquad W=2R(q+1),
\]
with the following property. If $c:\ZZ^r\to\{1,\ldots,q\}$ is a proper
coloring of $H_R$, define
\[
 \psi_c(g)=g+\mathcal A\bigl((c(g+v))_{v\in C_W}\bigr).
\]
Then
\begin{equation}\label{eq:lattice-count}
 \bigl|\{\psi_c(g):g\in C_R\}\bigr|\le3^r.
\end{equation}
\end{lemma}

\begin{proof}
Process the colors in the order $1,\ldots,q$. At stage $i$, select every
vertex of color $i$ having no previously selected neighbor. Vertices
selected at the same stage are nonadjacent, since the coloring is proper.
The resulting set $D$ is independent in $H_R$ and maximal. In particular,
every vertex is within distance $2R$ of a point of $D$.

Whether a vertex is selected depends only on the coloring within distance
$2Rq$ of it: each stage can enlarge the needed window by at most $2R$.
Choose $\psi_c(g)$ to be the selected point in $g+C_{2R}$ whose displacement
from $g$ is first in a fixed lexicographic order. This choice depends only
on the colors in $g+C_W$ and is translation-equivariant. It therefore
defines the rule $\mathcal A$. The same finite recursion defines the rule
on arbitrary input patterns; use displacement $0$ if no candidate is
selected. Its asserted properties are only needed for proper colorings.

For $g\in C_R$, the point $\psi_c(g)$ lies in $D\cap C_{3R}$. Partition
the integer interval $[-3R,3R]$ into
\[
 [-3R,-R],\qquad [-R+1,R],\qquad [R+1,3R].
\]
Their Cartesian products partition $C_{3R}$ into $3^r$ boxes, each of
$\ell^\infty$-diameter at most $2R$. Independence of $D$ permits at most one
point of $D$ in each box, proving~\eqref{eq:lattice-count}.
\end{proof}

The next lemma transfers this rule to a monoid action. Its conclusion
counts monoid elements, not merely their images in the action.

\begin{lemma}[Finite-pattern compression]\label{lem:compression}
Let $P\subseteq\ZZ^r$ generate $\ZZ^r$ as a group, where $r\ge1$, and let
$P\act X$ be bounded-to-one. For every finite $F\subseteq P$, there are
finite sets $K,L\subseteq P$ and a function $\lambda:X\to L$ such that
\begin{equation}\label{eq:compression}
 \left[p\mapsto px\text{ is injective on }K\right]
 \quad\Longrightarrow\quad
 \bigl|\{m+\lambda(mx):m\in F\}\bigr|\le3^r.
\end{equation}
\end{lemma}

\begin{proof}
Choose $R\ge1$ with $F\subseteq C_R$. By Lemma~\ref{lem:translate}, choose
$a\in P$ so that $a+C_{2R}\subseteq P$. Define a graph $H$ on $X$ by
joining distinct points of the form
\begin{equation}\label{eq:auxiliary-edges}
 az,\ (a+v)z
 \qquad(z\in X,\ v\in C_{2R}\setminus\{0\}).
\end{equation}
This graph has uniformly bounded degree. More explicitly, if $k_p$ bounds
the fibers of $p$, then its degree is at most
\[
 (|C_{2R}|-1)k_a+
 \sum_{v\in C_{2R}\setminus\{0\}}k_{a+v}.
\]
Choose $q$ larger than this bound and at least $(4R+1)^r$, and fix a proper
$q$-coloring $c$ of $H$. Such a coloring exists by greedy coloring of a
well-ordering of the vertices. No measurability is being asserted here.

Let $W$ and $\mathcal A$ be supplied by Lemma~\ref{lem:lattice-rule}, and
put $V=R+W$. Choose $b\in P$ such that
\begin{equation}\label{eq:positive-window}
 b+C_V-a\subseteq P.
\end{equation}
Set $K=b+C_V$ and $L=b+C_{2R}$. These are subsets of $P$, since $a\in P$.
Define on all of $X$
\begin{equation}\label{eq:lambda}
 \lambda(y)=b+
 \mathcal A\bigl((c((b+v)y))_{v\in C_W}\bigr).
\end{equation}
All action maps in this expression are defined, and its value belongs to
$L$.

Fix $x$ for which evaluation on $K$ is injective. The coloring
\[
 c_x(g)=c((b+g)x)\qquad(g\in C_V)
\]
is proper for the induced graph $H_R\upharpoonright C_V$. Indeed, for
adjacent $g,h\in C_V$, let $v=h-g$ and
$z=(b+g-a)x$. By~\eqref{eq:positive-window}, this is a legitimate point of
the action, and
\[
 (b+g)x=az,\qquad (b+h)x=(a+v)z.
\]
The two points are distinct by injectivity on $K$, so they are adjacent in
$H$ and receive different colors.

Because $q$ exceeds the degree of $H_R$, extend $c_x$ to a proper
$q$-coloring $\widetilde c_x$ of all of $\ZZ^r$, by greedily coloring the
remaining countably many vertices. For $m\in F$, the window $m+C_W$ lies
inside $C_V$. Commutativity and~\eqref{eq:lambda} therefore give
\[
 m+\lambda(mx)
 =b+m+\mathcal A\bigl((\widetilde c_x(m+v))_{v\in C_W}\bigr)
 =b+\psi_{\widetilde c_x}(m).
\]
Now $F\subseteq C_R$, so~\eqref{eq:lattice-count} proves
\eqref{eq:compression}.
\end{proof}

\begin{remark}\label{rem:parameters}
There is no circular choice of parameters in this construction. The graph
$H$ and its number of colors $q$ are chosen before the larger window $C_V$
and its translate $b$. Enlarging the window does not require recoloring a
larger auxiliary graph. The number $q$ affects displacement bounds, but
not the multiplicity bound $3^r$.
\end{remark}

\begin{corollary}[Local-freeness witness]\label{cor:local}
Let $P\subseteq\ZZ^r$ be finitely generated and generate $\ZZ^r$ as a
group, with $r\ge1$. For a bounded-to-one action $P\act X$ and every
integer $s\ge1$, there are a finite $K\subseteq P$, a finite $B$, and a
function $f:X\to X$ such that
\begin{align}
 d_X(y,f(y))&\le B &&(y\in X),\label{eq:displacement}\\
 p\mapsto px\text{ injective on }K
 &\quad\Longrightarrow\quad |f[B_X(x,s)]|\le3^r.
 \label{eq:local-multiplicity}
\end{align}
\end{corollary}

\begin{proof}
Let $S$ be the chosen finite generating set and put $h=\sum_{u\in S}u$.
For $u\in S$ we have $h-u\in P$. Consequently, if $\ell_S(b)\le s$, then
$sh-b\in P$. The set
\[
 F=\{a+sh-b:a,b\in P,\ \ell_S(a)+\ell_S(b)\le s\}
\]
is therefore a finite subset of $P$. The common-future formula implies
\begin{equation}\label{eq:ball-forward}
 (sh)B_X(x,s)\subseteq Fx\qquad(x\in X):
\end{equation}
if $ax=by$ with $\ell_S(a)+\ell_S(b)\le s$, apply $sh-b$ to obtain
$(sh)y=(a+sh-b)x$.

Apply Lemma~\ref{lem:compression} to $F$, obtaining $K,L,\lambda$. Define
\[
 f(y)=\lambda((sh)y)((sh)y).
\]
It moves points a distance at most
$\ell_S(sh)+\max_{\ell\in L}\ell_S(\ell)$. If evaluation on $K$ is
injective at $x$, then~\eqref{eq:ball-forward} gives
\[
 f[B_X(x,s)]
 \subseteq\{(m+\lambda(mx))x:m\in F\},
\]
which has at most $3^r$ elements by~\eqref{eq:compression}.
\end{proof}

\section{Induction on rank}\label{sec:induction}

\begin{proof}[Proof of Theorem~\ref{thm:main}]
We induct on $r$, proving the assertion simultaneously for all finitely
generated commutative monoids of that rank and all their bounded-to-one
actions. By Corollary~\ref{cor:lattice-reduction}, it is enough to consider
$P\subseteq\ZZ^r$ generating $\ZZ^r$ as a group. When $r=0$, this monoid
is trivial and all Schreier components are singletons, so $N_0=0$ works.

Let $r\ge1$ and assume the result in rank $r-1$. Fix an integer $s\ge1$.
Apply Corollary~\ref{cor:local} at scale $2s$, obtaining $K,B,f$. Put
\begin{equation}\label{eq:bad-good}
 A=\bigcup_{\substack{p,q\in K\\p\ne q}}A_{p,q},
 \qquad U=X\setminus A.
\end{equation}
Thus every point of $U$ passes the injectivity test on $K$.

\paragraph{The collision part.}
Each $A_{p,q}$ is forward invariant, so its intrinsic Schreier metric is
the restricted metric. Since $p-q\ne0$ in $\ZZ^r$,
Lemma~\ref{lem:collision} and the induction hypothesis imply
\[
 \asdim A_{p,q}\le N_{r-1}.
\]
There are only finitely many pairs in~\eqref{eq:bad-good}. The finite-union
theorem gives $\asdim A\le N_{r-1}$. Choose a uniformly bounded partition
$\cP_A$ of $A$ such that every $2s$-ball in the subspace $A$ meets at most
$N_{r-1}+1$ parts. If $A$ is empty, take the empty partition.

\paragraph{The locally free part.}
Partition $U$ into the nonempty fibers of $f\upharpoonright U$, and call
this partition $\cP_U$. By~\eqref{eq:displacement}, its parts have diameter
at most $2B$. For each $u\in U$, equation~\eqref{eq:local-multiplicity} at
scale $2s$ says that $B_X(u,2s)$ meets at most $3^r$ parts of $\cP_U$.

\paragraph{Recombining the two parts.}
The family $\cP_A\cup\cP_U$ is a uniformly bounded partition of $X$.
Consider $B_X(x,s)$. If it meets $A$, choose $a\in A\cap B_X(x,s)$. Then
\[
 B_X(x,s)\cap A\subseteq B_A(a,2s),
\]
so this portion meets at most $N_{r-1}+1$ parts. If it meets $U$, choose
$u\in U\cap B_X(x,s)$; similarly,
\[
 B_X(x,s)\cap U\subseteq B_X(u,2s),
\]
so this portion meets at most $3^r$ parts. In total the ball meets at most
\[
 (N_{r-1}+1)+3^r=N_r+1
\]
parts. Since $s$ was arbitrary, $\asdim X\le N_r$.
\end{proof}

\begin{remark}[Uniformity and quantifiers]\label{rem:quantifiers}
The finite set $K$, the collision sets in~\eqref{eq:bad-good}, and all
diameter bounds may depend on the scale, the monoid, the action, and its
fiber bounds. This is permitted in the definition of asymptotic dimension.
The \emph{number of parts that a ball may meet} is bounded solely in terms
of rank. At a fixed scale, the finite-union theorem takes a maximum over
finitely many diameter bounds; no uniform estimate over all possible
quotient monoids or all collision pairs is needed.
\end{remark}

\noindent\begin{minipage}{\textwidth}
\small
\textsc{School of Mathematical Sciences and School of Pre-university,}\\
\textsc{Dalian Minzu University}\\[0.4em]
\textit{Email address:}\enspace
\href{mailto:wangruijun@dlnu.edu.cn}{\texttt{wangruijun@dlnu.edu.cn}}
\end{minipage}
\end{document}